\documentclass[12pt,reqno]{amsart}
\usepackage{amscd,amsmath,amsthm,amssymb}
\usepackage{color}
\usepackage{tikz}
\usepackage{url}

\newtheorem{Theorem}{Theorem}[section]
\newtheorem{Lemma}[Theorem]{Lemma}
\newtheorem{Corollary}[Theorem]{Corollary}
\newtheorem{Proposition}[Theorem]{Proposition}
\newtheorem{Remark}[Theorem]{Remark}

\newtheorem{Example}[Theorem]{Example}

\newtheorem{Conjecture}[Theorem]{Conjecture}
\newtheorem{Question}[Theorem]{Question}

\def\qed{\ifhmode\textqed\fi
	\ifmmode\ifinner\quad\qedsymbol\else\dispqed\fi\fi}
\def\textqed{\unskip\nobreak\penalty50
	\hskip2em\hbox{}\nobreak\hfill\qedsymbol
	\parfillskip=0pt \finalhyphendemerits=0}
\def\dispqed{\rlap{\qquad\qedsymbol}}

\def\m{\mathfrak{m}}
\def\n{\mathfrak{n}}
\def\height{\textup{height}}
\def\Ass{\textup{Ass}}

\def\FF{\mathbb{F}}
\def\Tor{\textup{Tor}}
\def\depth{\textup{depth\,}}
\def\lcm{\textup{lcm}}
\def\im{\textup{im}}
\def\bicosize{\textup{big-cosize}}
\def\pd{\textup{proj\,dim}}
\def\supp{\textup{supp}}

\begin{document}
	
	\title{Monomial ideals whose all matching powers\\ are Cohen-Macaulay}
	\author{Antonino Ficarra, Somayeh Moradi}
	
	\address{Antonino Ficarra, Departamento de Matem\'{a}tica, Escola de Ci\^{e}ncias e Tecnologia, Centro de Investiga\c{c}\~{a}o, Matem\'{a}tica e Aplica\c{c}\~{o}es, Instituto de Investiga\c{c}\~{a}o e Forma\c{c}\~{a}o Avan\c{c}ada, Universidade de \'{E}vora, Rua Rom\~{a}o Ramalho, 59, P--7000--671 \'{E}vora, Portugal}
	\email{antonino.ficarra@uevora.pt}\email{antficarra@unime.it}
	
	\address{Somayeh Moradi, Department of Mathematics, Faculty of Science, Ilam University, P.O.Box 69315-516, Ilam, Iran}
	\email{so.moradi@ilam.ac.ir}
	
	%\thanks{}
	
	\begin{abstract}
		In the present paper, we aim to classify monomial ideals whose all matching powers are Cohen-Macaulay. We especially focus our attention on edge ideals. The Cohen-Macaulayness of the last matching power of an edge ideal is characterized, providing an algebraic analogue of the famous Tutte theorem regarding graphs having a perfect matching. For chordal graphs, very well-covered graphs and Cameron-Walker graphs, we completely solve our problem.
	\end{abstract}
	
	\maketitle
	\section*{Introduction}
	
	Let $S$ be either a regular local ring, or the polynomial ring $S=K[x_1,\dots,x_n]$ over a field $K$, and let $I\subset S$ be an ideal. We say that $I$ is Cohen-Macaulay if $S/I$ is a Cohen-Macaulay ring. Assume, moreover, that $I$ is radical. Then, a classical result in commutative algebra guarantees that $I$ is a complete intersection if and only if $I^k$ is Cohen-Macaulay for infinitely many $k$. This result was proved in the late seventies, more or less independently, by Achilles and Vogel in \cite{AV}, by Cowsik and Nori in \cite{CN}, and by Waldi in \cite{Waldi}.
	
	In particular, let $I=I(G)$ be the edge ideal of a finite simple graph $G$ on the vertex set $\{x_1,\dots,x_n\}$, that is, the ideal in the polynomial ring $S=K[x_1,\dots,x_n]$ generated by those monomials $x_ix_j$ for which $\{x_i,x_j\}$ is an edge. Hereafter, we will say that $G$ is a Cohen-Macaulay graph, if $I(G)$ is Cohen-Macaulay. We have that $I(G)$ is a radical ideal, and the aforementioned result implies that all powers $I(G)^k$ are Cohen-Macaulay if and only if $E(G)$ consists of disjoint edges. 
	
	The situation becomes more interesting if one turns its attention to other kind of powers. The first paper in this direction \cite{MT}, written by Minh and Trung, characterizes those unmixed 2-dimensional monomial ideals $I\subset S$ whose (almost) all symbolic powers $I^{(k)}$ are Cohen-Macaulay. After the appearance of this paper, many researchers tried to characterize all those unmixed monomial ideals $I\subset S$ whose (almost) all symbolic powers are Cohen-Macaulay, culminating with the beautiful solution due, independently, to Minh and Trung \cite{MTb}, and Varbaro \cite{Var}. Writing any squarefree monomial ideal $I$ as the Stanley-Reisner ideal $I_\Delta$ of a simplicial complex $\Delta$, the aforementioned authors proved, independently, that $I_\Delta^{(k)}$ is Cohen-Macaulay for all $k\gg0$ if and only if $\Delta$ is a matroid.
	
	Very recently, in \cite{BHZN18} and \cite{EHHM2022a}, the concept of squarefree power was introduced. Subsequently, this concept was extended to all monomial ideals with the introduction of \textit{matching powers} \cite{EF}, see also \cite{CFL3,DRS24,EF1,EH2021,EHHM2022b,FPack2,FHH2022,KNQ24,SASF2024,SASF2022,SASF2023,SASF2024b}. Hereafter, we always refer to this generalized concept. Let $I\subset S$ be a monomial ideal and let $\mathcal{G}(I)$ be its minimal monomial generating set. The $k$th matching power of $I$ is the monomial ideal generated by the products $u_1\cdots u_k$, where the $u_i\in\mathcal{G}(I)$ form a regular sequence. In other words, $\supp(u_i)\cap\supp(u_j)=\emptyset$ for all $1\le i<j\le k$, where the \textit{support} of a monomial $w\in S$ is defined as the set $\supp(w)=\{x_i:x_i\ \textup{divides}\ w\}$. Moreover, the \textit{monomial grade} of $I$ is defined as $\nu(I)=\max\{k:I^{[k]}\ne0\}$. A {\em matching} of $G$ is a family of pairwise disjoint edges of $G$. Let $\nu(G)$ be the maximum size of a matching of $G$. Then, it is easy to see that $\nu(I(G))=\nu(G)$. Furthermore, the generators of $I(G)^{[k]}$ correspond to the vertex sets of the $k$-matchings of $G$, which motivates the choice to name the ideals $I^{[k]}$ the matching powers of $I$. Our goal is to identify and classify those monomial ideals $I\subset S$ whose matching powers $I^{[k]}$ are Cohen-Macaulay, for all $1\le k\le\nu(I)$.
	
	Recently, in \cite{DRS24}, Das, Roy and Saha proved that $I(G)^{[k]}$ is Cohen-Macaulay for all $1\le k\le\nu(G)$, if $G$ is a Cohen-Macaulay forest. Moreover, it is proven in \cite{EF} and \cite{EF1} that $I(G)^{[\nu(G)]}$ is a polymatroidal ideal for any graph $G$. By a result due to Herzog and Hibi \cite{HH2006}, the Cohen-Macaulay polymatroidal ideals are of very special nature. These two facts suggest that there should be a pleasant theory regarding the Cohen-Macaulayness of the matching powers of an edge ideal. This is what we consider in this paper.
	More precisely, we study when all matching powers of $I(G)$ are Cohen-Macaulay. We characterize such graphs $G$ when $G$ has a perfect matching and the $(\nu(G)-1)$th matching power of $I(G)$ has height $2$. These conditions are fulfilled when $G$ is a Cohen-Macaulay very well-covered graph. Chordal graphs and Cameron-Walker graphs are other families for which we give such characterization. We will also discuss the situation for the graphs having up to $7$ vertices.
	\vspace{0.1cm}

	The paper is structured as follows.
	In Section \ref{sec1}, we investigate the hereditary nature of the Cohen-Macaulay property of $I(G)^{[k]}$.
	It is shown that if $I(G)^{[k]}$ possesses the Cohen-Macaulay property, then so does $I(G')^{[k]}$, where $G'$ is either the subgraph $G\setminus N_G[x]$ for any vertex $x$ of $G$ or a connected component of $G$, see Theorem~\ref{neighbourhood} and Corollary~\ref{ConnCom}. Another main result in this section is Theorem~\ref{Thm:I(G)[nu(G)]CM}, which gives a purely combinatorial characterization for the edge ideals whose last non-vanishing matching power is Cohen-Macaulay. More precisely, it is shown that $I(G)^{[\nu(G)]}$ is Cohen-Macaulay if and only if either $G$ has a perfect matching, or $G\setminus\{i\}$ has a perfect matching of size $|V(G)|-1$ for all $i\in V(G)$. This result can be seen as an algebraic analogue of the famous Tutte theorem in classical graph theory \cite{Tutte}. This characterization plays an important role in the study of edge ideals of graphs whose all matching powers are Cohen-Macaulay. 
	
	In Section~\ref{sec2} we concentrate our attention on the family of graphs which admit a perfect matching. In Theorem~\ref{Thm:I(G)[k]CM-perfect} we characterize the graphs $G$ from this family for which all matching powers $I(G)^{[k]}$ are Cohen-Macaulay, provided that $\dim S/I(G)^{[\nu(G)-1]}=|V(G)|-2$. This dimension can take only the two values $|V(G)|-2$ and $|V(G)|-3$, as is shown in Lemma~\ref{dimineq}. To prove Theorem~\ref{Thm:I(G)[k]CM-perfect} we use the Hilbert-Burch Theorem and a well-known result from graph theory on the characterization of trees. An interesting family of graphs which meets the conditions of Theorem~\ref{Thm:I(G)[k]CM-perfect} is the family of the Cohen-Macaulay very well-covered graphs. Using a characterization of this family due to Crupi, Rinaldo and Terai \cite{CRT2011}, in Proposition~\ref{Prop:dim-I(G)[k]-VWC} we give a formula for the dimension of $S/I(G)^{[k]}$ for all $k$. In particular, when $k=\nu(G)-1$, this dimension is $|V(G)|-2$. So, applying our main theorem to this family we get Corollary~\ref{Thm:I(G)[k]CM-VWC}, which shows that the only very well-covered graphs for which all matching powers of their edge ideals are Cohen-Macaulay, are only the Cohen-Macaulay forests. Bipartite graphs and whisker graphs are subfamilies of very well-covered graphs. Applying our result to these families we obtain the analogous characterizations, see Corollaries \ref{bipartite}, \ref{whisker}. We close this section by obtaining the normalized depth function of edge ideals of very well-covered graphs $G$ such that $I(G)^{[k]}$ is Cohen-Macaulay, for all $1\le k\le\nu(G)$, see Corollary~\ref{NorDepthFun}.
	
	In Section~\ref{sec3}, we answer our question on Cohen-Macaulay matching powers for the family of chordal graphs and Cameron-Walker graphs. For chordal graphs we use a result by Herzog, Hibi and Zheng~\cite{HHZ} which characterizes Cohen-Macaulay chordal graphs. In Theorem~\ref{chordal} we prove that all matching powers of the edge ideal of a chordal graph are Cohen-Macaulay if and only if $G$ is either a complete graph or a Cohen-Macaulay forest. Using the characterizations given for chordal graphs and bipartite graphs, in Theorem~\ref{Cameron-Walker} we answer our question for the family of Cameron-Walker graphs.
	
	In the final Section \ref{sec4}, by using the \textit{Macaulay2} \cite{GDS} packages \texttt{NautyGraphs} \cite{MPNauty} and \texttt{MatchingPowers} \cite{FPack2}, we determined, up to isomorphism, all graphs $G$ with a small number of vertices for which $I(G)^{[k]}$ is Cohen-Macaulay for all $1\le k\le\nu(G)$. Our experiments would indicate that among the graphs on $n\ge5$ non-isolated vertices, besides the complete graphs and the Cohen-Macaulay forests, there are only two extra graphs, up to isomorphism, for which the matching powers of their edge ideal is Cohen-Macaulay for all $k$. Furthermore, we expect that such a classification does not depend on the underlying field $K$ of the polynomial ring $S=K[x_i:i\in V(G)]$.
	
	\section{Cohen-Macaulayness of matching powers of edge ideals}\label{sec1}
	Throughout this paper, $G$ is a finite simple graph. Since the edge ideal of $G$ does not change by removing isolated vertices, we may extra assume that $G$ has no isolated vertices.  We denote the vertex set of $G$ by $V(G)$ and the edge set by $E(G)$.
	In this section we show that the Cohen-Macaulayness of a matching power of $I(G)$ is inherited by its connected components and by its subgraphs $G\setminus N_G[x]$ for any vertex $x$ of $G$. Moreover, we characterize those edge ideals whose last non-vanishing matching power is Cohen-Macaulay.
	
	Recall that for a vertex $x_i\in V(G)$, the sets
	\begin{align*}
		N_G(x_i)\ &=\ \{x_j\in V(G)\ :\ \{x_i,x_j\}\in E(G)\},\\
		N_G[x_i]\ &=\ N_G(x_i)\cup\{x_i\},
	\end{align*}
	are called the \textit{open neighbourhood} and the \textit{closed neighbourhood} of $x_i$ in $G$.
	
	The following lemma is needed for the proof of Theorem~\ref{neighbourhood}.
	
	\begin{Lemma}\label{CHHKTT}
		Let $I\subset S$ be a monomial ideal and $u\in S$ be a monomial with $u\notin I$. If $S/I$ is Cohen-Macaulay, then $S/(I:u)$ is Cohen-Macaulay and moreover $\depth S/I=\depth S/(I:u)$.
	\end{Lemma}
	
	\begin{proof}
		By \cite[Lemma 4.1]{CHHKTT}, $\depth S/I\leq \depth S/(I:u)$. Combining this with the inequalities $\depth S/(I:u)\leq \dim S/(I:u)\leq \dim S/I$ implies the assertion.
	\end{proof}
	
	For what follows, we need to recall some concepts and results from \cite{EK,FHT2009}. Let $I\subset S$ be a monomial ideal. By $\mathcal{G}(I)$ we denote the unique minimal set of monomial generators of $I$. Let $J,L\subset S$ be monomial ideals such that $\mathcal{G}(I)$ is the disjoint union of $\mathcal{G}(J)$ and $\mathcal{G}(L)$. Following \cite{FHT2009}, we say that $I=J+L$ is a \textit{Betti splitting}  if
	\begin{equation}\label{eq:BettiSplitEq}
		\beta_{i,j}(I)=\beta_{i,j}(J)+\beta_{i,j}(L)+\beta_{i-1,j}(J\cap L),\ \ \ \text{for all}\,\ i,j\ge0.
	\end{equation}
	
	Consider the natural short exact sequence
	\begin{equation}\label{eq:exactSequence-Betti}
		0\rightarrow J\cap L\rightarrow J\oplus L\rightarrow I\rightarrow0.
	\end{equation}
	By \cite[Proposition 2.1]{FHT2009}, equation (\ref{eq:BettiSplitEq}) holds if and only if the induced maps 
	$$
	\Tor_i^S(K,J\cap L)\rightarrow\Tor_i^S(K,J)\oplus\Tor_i^S(K,L)
	$$
	in the $\Tor$ exact sequence arising from the above sequence are zero for all $i\ge0$. In this case, we say that the inclusion map $J\cap L\rightarrow J\oplus L$ is \textit{$\Tor$-vanishing}. The next result appeared implicitly firstly in \cite[Proposition 3.1]{EK} (see also \cite[Lemma 4.2]{NV19}).
	
	\begin{Proposition}\label{Prop:CriterionBettiSplit}
		Let $J,L\subset S$ be non-zero monomial ideals with $J\subset L$. Suppose there exists a map $\varphi:\mathcal{G}(J)\rightarrow\mathcal{G}(L)$ such that for any $\emptyset\ne\Omega\subseteq\mathcal{G}(J)$ we have
		$$
		\lcm(u:u\in\Omega)\in \mathfrak{m}(\lcm(\varphi(u):u\in\Omega)),
		$$
		where $\mathfrak{m}=(x_1,x_2,\dots,x_n)$. Then the inclusion map $J\rightarrow L$ is $\Tor$-vanishing.
	\end{Proposition}
	
	For a monomial ideal $I\subset S$, we denote by $\partial^* I$ the ideal generated by the elements of the form $u/x_i$, where $u\in\mathcal{G}(I)$, and $x_i$ is a variable dividing $u$. The following result was proved in \cite[Proposition 4.4]{NV19}.
	
	\begin{Lemma}\label{Lemma:EsistenzaVarphi}
		Let $J,L\subset S$ be non-zero monomial ideals. Suppose that $\partial^* J\subseteq L$. Then $J\subseteq\mathfrak{m}L$ and there exists a map $\varphi:\mathcal{G}(J)\rightarrow \mathcal{G}(L)$ satisfying the assumption of Proposition \ref{Prop:CriterionBettiSplit}. In particular, the inclusion map $J\rightarrow L$ is $\Tor$-vanishing.
	\end{Lemma}

	\begin{Lemma}\label{Lem:Betti-Saha}
		Let $G$ be a graph, $1\le k\le\nu(G)$ be an integer, and $x\in V(G)$ be a vertex. Then,
		\begin{equation}\label{eq:BettiSaha}
			(I(G)^{[k]}:x)=\sum_{y\in N_G(x)} yI(G\setminus \{x,y\})^{[k-1]} + I(G\setminus N_G[x])^{[k]}.
		\end{equation}
		If $I(G\setminus N_G[x])^{[k]}\neq 0$, then this decomposition is a Betti splitting, and in particular,
		$$
		\depth S/I(G\setminus N_G[x])^{[k]}-|N_G(x)|\ \ge\ \depth S/(I(G)^{[k]}:x). 
		$$
	\end{Lemma}
	\begin{proof}
		The equation (\ref{eq:BettiSaha}) was proved in \cite[Lemma 2.9(ii)]{DRS24}. To simplify our discussion, we set $I=(I(G)^{[k]}:x)$, $J=\sum_{y\in N_G(x)} yI(G\setminus \{x,y\})^{[k-1]}$, $L=I(G\setminus N_G[x])^{[k]}$, and we assume, up to a relabeling, that $N_G(x)=\{y_1,\dots,y_m\}$.
		
		Firstly, we notice that $\mathcal{G}(I)$ is the disjoint union of $\mathcal{G}(J)$ and $\mathcal{G}(L)$. Indeed, all the minimal generators of $J$ are monomials of degree $2k-1$ and are divided by some variable $y_i$, whereas each minimal generator of $L$ is a monomial of degree $2k$ and is not divided by $y_i$, for all $1\le i\le m$.
		
		Now, we compute $J\cap L$. We have
		\begin{align*}
			J\cap L\ &=\ (\sum_{i=1}^m y_iI(G\setminus \{x,y_i\})^{[k-1]})\cap I(G\setminus N_G[x])^{[k]}\\
			&=\ \sum_{i=1}^my_i[I(G\setminus \{x,y_i\})^{[k-1]}\cap I(G\setminus N_G[x])^{[k]}]\\
			&=\ \sum_{i=1}^my_iI(G\setminus N_G[x])^{[k]}\ =\ (y_1,\dots,y_m)I(G\setminus N_G[x])^{[k]}.
		\end{align*}
		
		Here, in the second equality we used that $y_i$ does not divide any generator of $I(G\setminus N_G[x])^{[k]}$, and in the third equality we used, for $1\le i\le m$, the chain of inclusions $I(G\setminus N_G[x])^{[k]}\subset I(G\setminus N_G[x])^{[k-1]}\subseteq I(G\setminus\{x,y_i\})^{[k-1]}$.
		
		To prove that the inclusion map $J\cap L\rightarrow J\oplus L$ is $\Tor$-vanishing, it is enough to show that the inclusion maps $J\cap L\rightarrow J$ and $J\cap L\rightarrow L$ are $\Tor$-vanishing.
		
		Notice that $\partial^* I(G\setminus N_G[x])^{[k]}\subset I(G\setminus N_G[x])^{[k-1]}\subseteq \sum_{i=1}^mI(G\setminus\{x,y_i\})^{[k-1]}$. Therefore, by Lemma \ref{Lemma:EsistenzaVarphi}, there exists a map
		$$
		\varphi\ :\ \mathcal{G}(I(G\setminus N_G[x])^{[k]})\rightarrow\mathcal{G}(\sum_{i=1}^mI(G\setminus\{x,y_i\})^{[k-1]})
		$$
		such that for every non-empty subset $\Omega\subseteq\mathcal{G}(I(G\setminus N_G[x])^{[k]})$ we have
		$$
		\lcm(u:u\in\Omega)\in\m(\lcm(\varphi(u):u\in\Omega)),
		$$
		where $\m$ is the graded maximal ideal of $S$.
		
		Each minimal generator of $J\cap L=(y_1,\dots,y_m)L$ is of the form $y_iu$ with $1\le i\le m$ and $u\in\mathcal{G}(I(G\setminus N_G[x])^{[k]})$. Now we define a map $\Phi:\mathcal{G}(J\cap L)\rightarrow\mathcal{G}(J)$ by setting $\Phi(y_iu)=y_i\varphi(u)$. It is clear that $\Phi(y_iu)\in\mathcal{G}(J)$. Let $\Omega=\{y_{i_\ell}u_\ell\}_\ell\subseteq\mathcal{G}(J\cap L)$ be non-empty. Then
		\begin{align*}
			\lcm(\{y_{i_\ell}u_\ell\}_\ell)=\lcm(\{y_{i_\ell}\}_\ell)\lcm(\{u_\ell\}_\ell)\ &\in\ \lcm(\{y_{i_\ell}\}_\ell)\cdot\m(\lcm(\{\varphi(u_\ell)\}_\ell))\\
			&=\ \m(\lcm(\{\Phi(y_{i_\ell}u_{\ell})\}_\ell)).
		\end{align*}
		Thus, Proposition \ref{Prop:CriterionBettiSplit} implies that $J\cap L\rightarrow J$ is $\Tor$-vanishing.
		
		To prove that $J\cap L\rightarrow L$ is $\Tor$-vanishing, notice that $J\cap L=(y_1,\dots,y_m)L$ and that the variables $y_i$ do not divide any minimal generator of $L$. Therefore, we can define a map $\Psi:\mathcal{G}(J\cap L)\rightarrow\mathcal{G}(L)$ by setting $\Psi(y_iu)=u$ for all $1\le i\le m$ and all $u\in\mathcal{G}(L)$. It is clear that this map satisfies the condition in Proposition \ref{Prop:CriterionBettiSplit}, and so the map $J\cap L\rightarrow L$ is $\Tor$-vanishing.
		
		Hence, equation (\ref{eq:BettiSaha}) is indeed a Betti splitting. It follows from equation (\ref{eq:BettiSplitEq}) that
		$$
		\depth S/I\ =\ \min\{\depth S/J,\depth S/L,\depth S/(J\cap L)-1\}.
		$$
		
		Since $J\cap L=(y_1,\dots,y_m)L$, we have $\depth S/(J\cap L)=\depth S/L-(m-1)$. Therefore, it follows that
		$$
		\depth S/L-m=\depth S/(J\cap L)-1\ge\depth S/I.
		$$
		Since $m=|N_G(x)|$, the proof is complete.
	\end{proof}
	
	As a consequence, we obtain the following useful permanence property.
	
	\begin{Theorem}\label{neighbourhood}
		Let $G$ be a graph and $1\le k\le\nu(G)$ be an integer. If $I(G)^{[k]}$ is Cohen-Macaulay, then $I(G\setminus N_G[x])^{[k]}$ is Cohen-Macaulay, for any vertex $x\in V(G)$.
	\end{Theorem}
	
	\begin{proof}
		Consider a vertex $x\in V(G)$. We may assume that $I(G\setminus N_G[x])^{[k]}\neq 0$. Then by Lemma \ref{Lem:Betti-Saha} we have $$\depth S/I(G\setminus N_G[x])^{[k]}-|N_G(x)|\geq  \depth S/(I(G)^{[k]}:x).$$ Since $I(G)^{[k]}$ is Cohen-Macaulay, by Lemma~\ref{CHHKTT} we have 
		$$
		\depth S/(I(G)^{[k]}:x)= \depth S/I(G)^{[k]}.
		$$
		
		From this fact and the assumption that $I(G)^{[k]}$ is Cohen-Macaulay, we get that $\dim S/I(G\setminus N_G[x])^{[k]}-|N_G(x)|\geq \dim S/I(G)^{[k]}$. We show that indeed 
		\begin{equation}\label{eq:dimneighborhood}
			\dim S/I(G\setminus N_G[x])^{[k]}-|N_G(x)|=\dim S/I(G)^{[k]}.
		\end{equation}
		
		Let $N_G(x)=\{y_1,\ldots,y_m\}$, and let $P\subset S$ be a monomial prime ideal such that $I(G\setminus N_G[x])^{[k]}\subseteq P$. Then one can see that $I(G)^{[k]}\subset(y_1,\ldots,y_m,P)$. Since any minimal prime ideal of $I(G\setminus N_G[x])^{[k]}$ is a monomial ideal generated by variables, it follows that $\height\,I(G)^{[k]}\leq \height\,I(G\setminus N_G[x])^{[k]}+|N_G(x)|$. This shows the equality~(\ref{eq:dimneighborhood}).
		Thus, we obtain
		\begin{align*}
			\depth S/I(G\setminus N_G[x])^{[k]}\ &\ge\ \depth S/I(G)^{[k]}+|N_G(x)|\ =\ \dim S/I(G)^{[k]}+|N_G(x)|\\
			&=\ \dim S/I(G\setminus N_G[x])^{[k]}\ \ge\ \depth S/I(G\setminus N_G[x])^{[k]}.
		\end{align*}
		Hence $I(G\setminus N_G[x])^{[k]}$ is Cohen-Macaulay, as desired.
	\end{proof}

	The following result is an immediate consequence of Theorem~\ref{neighbourhood}.
	
	\begin{Corollary}\label{ConnCom}
		Let $G$ be a graph with the connected components $G_1,\ldots,G_m$, and let $1\le k\le\nu(G)$ be an integer. If $I(G)^{[k]}$ is Cohen-Macaulay, then $I(G_i)^{[k]}$ is Cohen-Macaulay for all $i$. 
	\end{Corollary}
	
	The converse of Corollary~\ref{ConnCom} does not hold in general.
	
	\begin{Example}
		\rm Let $G$ be a graph with the connected components $K_2$ and $K_3$, where $K_n$ denotes the complete graph on $n$ vertices. Then $I(G)$ is Cohen-Macaulay, but $I(G)^{[2]}$ is not Cohen-Macaulay, since $\dim S/I(G)^{[2]}=4$ and $\depth S/I(G)^{[2]}=3$. 
	\end{Example}
	
	Next, we characterize those edge ideals $I(G)$ having the property that their last non-vanishing matching power is Cohen-Macaulay.
	
	It is clear that for any graph $G$ on $n$ non-isolated vertices we have $\nu(G)\le\left\lfloor\frac{n}{2}\right\rfloor$.
	
	A {\em perfect matching} in  $G$ is a matching $M$ of $G$ such that each vertex of $G$ is an endpoint of some edge in $M$. So $G$ has a perfect matching if and only if $n=2\nu(G)$.
	
	\begin{Theorem}\label{Thm:I(G)[nu(G)]CM}
		Let $G$ be a graph on $n$ non-isolated vertices. The following conditions are equivalent.
		\begin{enumerate}
			\item[\textup{(a)}] $I(G)^{[\nu(G)]}$ is Cohen-Macaulay.
			\item[\textup{(b)}] $I(G)^{[\nu(G)]}=\m^{[2\nu(G)]}$ and $\nu(G)=\lfloor\frac{n}{2}\rfloor$.
			\item[\textup{(c)}] Either $G$ has a perfect matching, or $G\setminus\{i\}$ has a perfect matching of size $|V(G)|-1$ for all $i\in V(G)$.
		\end{enumerate}
	\end{Theorem}
	\begin{proof}
		(a) $\Rightarrow$ (b) By \cite[Theorem 1.7]{EF} (or also \cite[Theorem 1.1]{EF1}), $I(G)^{[\nu(G)]}$ is a squarefree polymatroidal ideal. By Herzog and Hibi \cite[Theorem 4.2]{HH2006} (see also \cite[Corollary 3]{CF2024}) the only squarefree Cohen-Macaulay polymatroidal ideals are the squarefree principal ideals and the squarefree Veronese ideals $\m^{[s]}$ for $1\le s\le n$. Then, our assumption implies that $I(G)^{[\nu(G)]}$ is either a squarefree principal ideal, or a squarefree Veronese ideal generated in degree $2\nu(G)$.
		
		Let $k=\nu(G)$. We claim that $k=\lfloor\frac{n}{2}\rfloor$ and $I(G)^{[k]}=\m^{[2k]}$.
		
		To this end, we first show that any variable divides some generator of $I(G)^{[k]}$. Up to relabeling, we may assume that $M=\{e_i:1\le i\le k\}$, with $e_i=\{2i-1,2i\}$ for all $i$, is a $k$-matching of $G$. If $2k=n$, the assertion is proved. Suppose $2k<n$. Thus it remains to show that $x_v$ divides some generator of $I(G)^{[k]}$ for all $2k< v\le n$. There exists $w\in V(G)$ with $\{v,w\}\in E(G)$. We have $w\in V(M)$, otherwise $M\cup\{\{v,w\}\}$ would be a $(k+1)$-matching of $G$, which is absurd because $\nu(G)=k$. Thus $w\in e_i$ for some $i$. Hence $M'=(M\setminus e_i)\cup\{\{v,w\}\}$ is a $k$-matching of $G$, and the generator of $I(G)^{[k]}$ corresponding to this matching is divided by $x_v$.\smallskip
		
		Now, if $2k=n$ then $I(G)^{[k]}=(x_1x_2\cdots x_n)=\m^{[n]}=\m^{[2k]}$ and $k=\frac{n}{2}=\lfloor\frac{n}{2}\rfloor$. Suppose now $2k<n$. Then, since all variables divide some generator of $I(G)^{[k]}$, it follows that $I(G)^{[k]}$ is not principal. Thus, by the first paragraph of the proof, $I(G)^{[k]}$ is the squarefree Veronese ideal $\m^{[2k]}$ of degree $2k$ in the variables $x_1,\dots,x_n$. It remains to show that $k=\lfloor\frac{n}{2}\rfloor$. Let $K_n$ be the complete graph on $n$ vertices. Then $I(G)\subseteq I(K_n)$ and $I(G)^{[k]}=\m^{[2k]}=I(K_n)^{[k]}$. Thus, \cite[Proposition 1.3]{EHHM2022b} implies that $I(G)^{[\ell]}=I(K_n)^{[\ell]}=\m^{[2\ell]}$ for all $\ell\ge k$. Since $k=\nu(G)$, we have $I(G)^{[\ell]}=0$ for all $\ell>k$. Hence $\m^{[2\ell]}=0$ for all $\ell>k$, and so $k=\lfloor\frac{\nu(\m)}{2}\rfloor=\lfloor\frac{n}{2}\rfloor$.\smallskip
		
		(b) $\Rightarrow$ (a) Conversely, if $\nu(G)=\lfloor\frac{n}{2}\rfloor$ and $I(G)^{[\nu(G)]}=\m^{[2\nu(G)]}$, then $I(G)^{[\nu(G)]}$ is Cohen-Macaulay by the result of Herzog and Hibi \cite[Theorem 4.2]{HH2006}.\smallskip
		
		(b) $\Leftrightarrow$ (c) Let $k=\lfloor\frac{n}{2}\rfloor$. Then either $n=2k$ or $n=2k+1$. Depending on these cases, we see that $I(G)^{[\nu(G)]}=\m^{[2k]}$ if and only if statement (c) holds.
	\end{proof}
	
	\section{Graphs with perfect matchings and very well-covered graphs}\label{sec2}
	
	To study our question on Cohen-Macaulay matching powers and taking Theorem~\ref{Thm:I(G)[nu(G)]CM} into account, it makes sense to consider the family of graphs which have a perfect matching. This is our focus in this section. It turns out that the Cohen-Macaulayness of the $(\nu(G)-1)$th matching power for this class of graphs is crucial in order to get the answer for all matching powers.
	Theorem~\ref{Thm:I(G)[k]CM-perfect} shows this fact. In order to prove our result we use the Hilbert-Burch Theorem and the next elementary result of graph theory.
	\begin{Lemma}\label{Lem:elementary}
		Let $G$ be a connected graph on $n$ non-isolated vertices. Then $$|E(G)|\ge n-1,$$ and equality holds if and only if $G$ is a tree.
	\end{Lemma}
	
	We now state and prove our main result in the section.
	\begin{Theorem}\label{Thm:I(G)[k]CM-perfect}
		Let $G$ be a graph with a perfect matching of size $n\ge2$ and such that $\dim S/I(G)^{[n-1]}=2n-2$. Then, the following conditions are equivalent.
		\begin{enumerate}
			\item[\textup{(a)}] $I(G)^{[k]}$ is Cohen-Macaulay, for all $1\le k\le\nu(G)=n$.
			%  \item[\textup{(b)}] $I(G)$ and $I(G)^{[n-1]}$ are Cohen-Macaulay.
			\item[\textup{(b)}] $G$ is a Cohen-Macaulay forest.
		\end{enumerate}
	\end{Theorem}
	
	\begin{proof}
		The implication (b) $\Rightarrow$ (a) is a particular case of \cite[Corollary 3.6]{DRS24}.  
		
		(a) $\Rightarrow$ (b) First we prove the assertion for the case that $G$ is a connected graph. 
		By assumption on the dimension we get $\height\,I(G)^{[n-1]}=2$. Thus, $I(G)^{[n-1]}$ is a perfect ideal of grade 2, and by the Hilbert-Burch Theorem \cite[Theorem 1.4.17]{BH}, the minimal (multi)graded free resolution of $I(G)^{[n-1]}$ is of the form
		$$
		\FF\ \ :\ \ 0\rightarrow S^{m-1}\xrightarrow{X}S^{m}\rightarrow I(G)^{[n-1]}\rightarrow 0,
		$$
		where $m$ is the minimal number of generators of $I(G)^{[n-1]}$, the matrix $X$ represents the homomorphism $S^{m-1}\rightarrow S^m$, and $I_{m-1}(X)=I(G)^{[n-1]}$. Here, for an integer $t$ and a matrix $A$, we denote by $I_t(A)$ the ideal generated by all $t$-minors of $A$. 
		
		We claim that $m\ge 2n-1$. To this end, we show that there is an injective map
		$$
		\Phi\ :\ E(G)\rightarrow\mathcal{G}(I(G)^{[n-1]}).
		$$
		
		By Lemma \ref{Lem:elementary}, $|E(G)|\ge|V(G)|-1=2n-1$. Thus if we show the existence of the map $\Phi$, then $m\ge|E(G)|\ge2n-1$ as desired.
		
		Let ${\bf x}=x_1x_2\cdots x_n$ and ${\bf y}=y_1y_2\cdots y_n$. 
		%By Theorem \ref{Thm:veryWellCGCM}, we do not have edges of the form $\{y_i,y_j\}$ for some $i\ne j$, but we have the edges $\{x_i,y_i\}$ for all $i$, and if $\{x_i,y_j\}$ is an edge of $G$ for some $i<j$, then $\{x_i,x_j\}$ is not an edge of $G$. In view of this fact, 
		We define $\Phi$ as follows:
		
		\begin{align*}
			\textup{(i)}&& \Phi(\{x_i,y_i\})\ &=\ ({\bf xy})/(x_iy_i)\ =\ \,\prod_{\ell\ne i}(x_\ell y_\ell)\in\mathcal{G}(I(G)^{[n-1]}),\\
			\textup{(ii)}&& \Phi(\{x_i,y_j\})\ &=\ ({\bf xy})/(y_ix_j)\ =\ (x_iy_j)\prod_{\ell\ne i,j}(x_\ell y_\ell)\in\mathcal{G}(I(G)^{[n-1]}),\\
			\textup{(iii)}&& \Phi(\{x_i,x_j\})\ &=\ ({\bf xy})/(y_iy_j)\ =\ (x_ix_j)\prod_{\ell\ne i,j}(x_\ell y_\ell)\in\mathcal{G}(I(G)^{[n-1]}),\\
			\textup{(iv)}&& \Phi(\{y_i,y_j\})\ &=\ ({\bf xy})/(x_ix_j)\ =\ (y_iy_j)\prod_{\ell\ne i,j}(x_\ell y_\ell)\in\mathcal{G}(I(G)^{[n-1]}).
		\end{align*}
		
		It is clear that the above monomials are pairwise distinct. Hence $\Phi$ is injective.
		
		Since $m\ge 2n-1$ and each non-zero entry of $X$ is an homogeneous element of degree at least one, we see that each non-zero $(m-1)$-minor of $X$ has degree at least $m-1\ge 2n-2$. Since $I(G)^{[n-1]}$ is equigenerated in degree $2(n-1)$, we then see that the equality $I_{m-1}(X)=I(G)^{[n-1]}$ implies that $m=2n-1$ and that each entry of $X$ is linear (that is, of degree one). By the inequalities $m\ge|E(G)|\ge2n-1$ we obtain $|E(G)|=2n-1$. So, Lemma \ref{Lem:elementary} implies that $G$ is a Cohen-Macaulay tree.
		
		Now, suppose that $G$ has connected components $G_1,\ldots,G_r$. We show that each $G_i$ is a tree. Since $G$ has a perfect matching, each $G_i$ has a perfect matching. Let $|V(G_i)|=2n_i$ for all $i$.
		We fix an integer $1\leq i\leq r$. If $n_i=1$, then $G_i$ is a tree, and we are done. So we may assume that $n_i\geq 2$.
		Set $J=\prod_{j=1,j\neq i}^r I(G_j)^{[n_j]}$. Then $J=(u)$ for some monomial $u$ of degree $2n-2n_i$ and we have the inclusion $uI(G_i)^{[n_i-1]}\subseteq I(G)^{[n-1]}$. Since $I(G)^{[n-1]}$ is Cohen-Macaulay with $\height\,I(G)^{[n-1]}=2$, we conclude that $I(G)^{[n-1]}=\bigcap_{\ell=1}^t P_\ell$, where each $P_\ell$ is a monomial prime ideal of height $2$. If $u\in P_\ell$ for all $\ell$, then $u\in I(G)^{[n-1]}$. This is not possible, since $I(G)^{[n-1]}$ is generated in degree $2n-2$ and $\deg u=2n-2n_i<2n-2$. So $u\notin P_\ell$ for some $\ell$. Since $uI(G_i)^{[n_i-1]}\subseteq P_\ell$, we obtain $I(G_i)^{[n_i-1]}\subseteq P_\ell$. This shows that   
		$\height\,I(G_i)^{[n_i-1]}\leq 2$,
		which implies that 
		$\dim S_i/I(G_i)^{[n_i-1]}\geq 2n_i-2$, where $S_i=K[V(G_i)]$. Since $n_i\geq 2$, by Lemma~\ref{dimineq} it follows that $\dim S_i/I(G_i)^{[n_i-1]}=2n_i-2$. Moreover, by Corollary~\ref{ConnCom} we conclude that $I(G_i)^{[k]}$
		is Cohen-Macaulay for all $1\leq k\leq n_i$. Therefore, by the first part of the proof on the connected case, it follows that $G_i$ is a Cohen-Macaulay tree. Hence, $G$ is a Cohen-Macaulay forest. 
	\end{proof}

	In the proof of Theorem~\ref{Thm:I(G)[k]CM-perfect} it is shown that for a connected graph $G$  with a perfect matching such that $S/I(G)^{[n-1]}$ is Cohen-Macaulay of dimension $2n-2$, the Hilbert-Burch matrix $X$ has only variables as non-zero entries. This gives the following
	
	\begin{Corollary}
		Let $G$ be a connected graph with a perfect matching of size $n$ such that $S/I(G)^{[n-1]}$ is Cohen-Macaulay of dimension $2n-2$. Then $I(G)^{[n-1]}$ has a linear resolution.
	\end{Corollary}
	
	The dimension of $S/I(G)^{[\nu(G)-1]}$ which is referred in the statement of Theorem~\ref{Thm:I(G)[k]CM-perfect} could get only two values as the following lemma shows.

	\begin{Lemma}\label{dimineq}
		Let $G$ be a graph with a perfect matching on $2n$ vertices. Then $$2k-1\leq \dim S/I(G)^{[k]}\leq n+k-1,$$
		for all $1\le k\le\nu(G)=n$.
		In particular, if $n\geq 2$, then
		$$\dim S/I(G)^{[n-1]}\in\{2n-3,2n-2\}.$$ 
	\end{Lemma}
	
	\begin{proof}
		Let $M=\{e_1,\ldots,e_n\}$ be a perfect matching of $G$ with $e_i=\{x_i,y_i\}$ for all $i$. For each non-empty 
		subset $F\subseteq [n]$, we set $u_F=\prod_{i\in F}(x_iy_i)$. Let $1\le k\le n$ be an integer and let $J=(u_F: F\subseteq [n],\, |F|=k)$.
		Since $J\subseteq I(G)^{[k]}$, to prove the right-hand side inequality, it is enough to show that $\height\, J\geq n-k+1$. If this was not the case, there would exists a monomial prime ideal $P$ such that $J\subseteq P$ and 
		$\height\, P\leq n-k$. Set $A=\{i\in [n]: x_i\in P \textup{ or } y_i\in P\}$. Without loss of generality assume that $A\subseteq [n-k]$. Then $u_F\notin P$ for $F=[n]\setminus [n-k]$, which is a contradiction. Finally, by \cite[Proposition 1.1]{EHHM2022b} we have $\depth S/I(G)^{[k]}\geq 2k-1$.
	\end{proof}
	\begin{Remark}
		\rm The inequality $2n-3\le\dim S/I(G)^{[n-1]}$ in the previous lemma, can also be obtained as follows. Recall that the \textit{big cosize} of a monomial ideal $J\subset S$ with minimal generating set $\mathcal{G}(J)=\{u_1,\dots,u_m\}$, is denoted by $\bicosize\,J$, and it is defined as the smallest integer $s$ such that
		$$
		\lcm(u_{j_1},u_{j_2},\dots,u_{j_s})\ =\ \lcm(u_1,u_2,\dots,u_m),
		$$
		for all integers $1\le j_1<j_2<\dots<j_s\le m$. It is shown in \cite[Theorem 3.3(a)]{ACF3} that $\pd\,S/J\le\bicosize\,J$. Notice that $I(G)^{[n-1]}$ is a monomial ideal equigenerated in degree $2n-2$ in the polynomial ring $S$ having $2n$ variables. From this observation it follows immediately that $\bicosize\,I(G)^{[n-1]}\le3$. Hence $\pd\,S/I(G)^{[n-1]}\le3$. Then, the Auslander-Buchsbaum formula implies that $2n-3\le\depth S/I(G)^{[n-1]}\le\dim S/I(G)^{[n-1]}$, as desired.
	\end{Remark}
	
	An interesting family of graphs $G$ which satisfy the assumption that $G$ has a perfect matching of size $n$ and $\dim S/I(G)^{[n-1]}=2n-2$ are Cohen-Macaulay very well-covered graphs. Before showing this fact, we recall some definitions. 
	
	A graph $G$ is called \emph{unmixed} or \emph{well--covered} if all the minimal vertex covers of $G$ have the same size. Equivalently, $I(G)$ is an unmixed ideal, see \cite{HHBook}. In \cite[Corollary 3.4]{GV}, Gitler and Valencia proved that for an unmixed graph $G$ we have
	$$
	\height\,I(G)\ge|V(G)|/2.
	$$
	
	We say that $G$ is \textit{very well-covered} if $G$ is unmixed and $\height\,I(G)=|V(G)|/2$.
	
	By a result of Favaron \cite[Theorem 1.2]{OF}, very well-covered graphs have always perfect matchings. Thus, by our Theorem \ref{Thm:I(G)[nu(G)]CM}, $I(G)^{[\nu(G)]}$ is Cohen-Macaulay, and we have $|V(G)|=2n$ for some $n\ge1$, for any very well-covered graph $G$.
	
	To classify those very well-covered graphs $G$  such that $I(G)^{[k]}$ is Cohen-Macaulay for all $1\le k\le\nu(G)$, we can use Theorem~\ref{Thm:I(G)[k]CM-perfect}, once we show that $\dim S/I(G)^{[n-1]}=2n-2$, where $n$ is the matching number of $G$. This equality is proved in Proposition~\ref{Prop:dim-I(G)[k]-VWC}. For the proof of this result, we will need the following fundamental classification due to Crupi, Rinaldo and Terai \cite{CRT2011}. Here, we use a formulation of this result given in \cite[Lemma 3.1]{MMCRTY2011} (see also \cite[Characterization 1.12]{CF1} and \cite{CF2}). Hereafter, we identify each vertex of the graph with the corresponding variable.
	
	\begin{Theorem}\label{Thm:veryWellCGCM}
		Let $G$ be a very well-covered graph. Then $|V(G)|=2n$ for some $n\ge1$ and the following conditions are equivalent.
		\begin{enumerate}
			\item[\textup{(a)}] $I(G)\subset S=K[x_1,\dots,x_n,y_1,\dots,y_n]$ is Cohen-Macaulay.
			\item[\textup{(b)}] There exists a relabeling of $V(G)=\{x_1,\dots,x_n,y_1,\dots,y_n\}$ such that:\smallskip
			\item[\textup{(i)}] $X=\{x_1,\dots,x_n\}$ is a minimal vertex cover of $G$, and\\ $Y=\{y_1,\dots,y_n\}$ is a maximal independent set of $G$,
			\item[\textup{(ii)}] $x_iy_i\in I(G)$ for all $1\le i\le n$,
			\item[\textup{(iii)}] if $x_iy_j\in I(G)$ then $i\le j$,
			\item[\textup{(iv)}] if $x_iy_j\in I(G)$ then $x_ix_j\notin I(G)$,
			\item[\textup{(v)}] if $z_ix_j,y_jx_k\in I(G)$ then $z_ix_k\in I(G)$ for any distinct $i,j,k$ and $z_i\in\{x_i,y_i\}$.
		\end{enumerate}
	\end{Theorem}
	
	Here is a typical example of a Cohen-Macaulay very well-covered graph. This graph has $8$ vertices, and it satisfies the conditions (i)-(v) of the previous result.
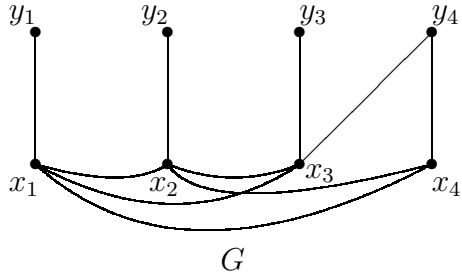
\vspace*{-0.6cm}
	\begin{center}
		\begin{picture}(90,150)(-0,-60)
			\put(-10,60){\circle*{4}}
			\put(40,60){\circle*{4}}
			\put(-20,65){\textit{$y_1$}}
			\put(30,65){\textit{$y_2$}}
			\put(-10,10){\circle*{4}}
			\put(40,10){\circle*{4}}
			\put(-20,0){\textit{$x_1$}}
			\put(33,0){\textit{$x_2$}}
			\put(90,10){\circle*{4}}
			\put(92,5){\textit{$x_3$}}
			\put(90,60){\circle*{4}}
			\put(90,65){\textit{$y_3$}}
			\put(-10,60){\line(0,-1){50}}
			\put(40,60){\line(0,-1){50}}
			\put(90,60){\line(0,-1){50}}
			\put(90,10){\line(1,1){50}}
			\put(140,0){\textit{$x_4$}}
			\put(140,10){\circle*{4}}
			\put(140,65){\textit{$y_4$}}
			\put(140,60){\circle*{4}}
			\put(140,60){\line(0,-1){50}}
			\qbezier(-10,10)(25,0)(40,10)
			\qbezier(-10,10)(45,-20)(90,10)
			\qbezier(-10,10)(45,-40)(140,10)
			\qbezier(40,10)(65,0)(90,10)
			\qbezier(40,10)(55,-12)(140,10)
			\put(60,-30){\textit{$G$}}
		\end{picture}
	\end{center}\vspace*{-0.9cm}

	\begin{Proposition}\label{Prop:dim-I(G)[k]-VWC}
		Let $G$ be a Cohen-Macaulay very well-covered graph. Then,
		$$\dim S/I(G)^{[k]}=|V(G)|/2+k-1\,\,\,\textit{for all}\,\,\,\,1\le k\le\nu(G).$$
	\end{Proposition}
	\begin{proof}
		We have $|V(G)|=2n$ and $\nu(G)=n$, for some $n\ge1$. By Lemma~\ref{dimineq}, the result follows, once we show that  $\dim S/I(G)^{[k]}\ge n+k-1$ for all $k$. Equivalently, we must prove that $\height\,I(G)^{[k]}\le n-k+1$ for all $k$. Consider the prime ideal $P=(x_1,x_2,\dots,x_{n-k+1})$ for all $1\le k\le n$. 
		It is enough to show that $I(G)^{[k]}\subseteq P$.
		%That is,
		%	\begin{enumerate}
			%	\item[\textup{(i)}] $I(G)^{[k]}\subseteq P$, but
			%		\item[\textup{(ii)}] $I(G)^{[k]}\not\subseteq P_i=(x_1,\dots,x_{i-1},x_{i+1},\dots,x_{n-k+1})$, for all $1\le i\le n-k+1$.
			%	\end{enumerate}
		Suppose the contrary. Then we could find a monomial generator $u\in\mathcal{G}(I(G)^{[k]})$ which is not divided by any of the variables $x_1,x_2,\dots,x_{n-k+1}$. Such an element $u$, if it exists, would be a squarefree monomial of degree $2k$ in the set of variables
		$$
		\{y_i:\ 1\le i\le n-k+1\}\cup\{x_j,y_j:\ n-k+2\le j\le n\}.
		$$
		Note that none of the variables $y_i$, with $1\le i\le n-k+1$, can divide $u$, because by Theorem \ref{Thm:veryWellCGCM}(iii) any edge containing $y_i$ must contain also a vertex from the set
		$$
		\{x_i:\ 1\le i\le n-k+1\},
		$$
		and then the corresponding variable $x_i\in P$ would divide $u$, which is forbidden. Hence, $u$ would be a squarefree monomial in the set of variables
		$$
		\{x_j,y_j:\ n-k+2\le j\le n\}.
		$$
		This would imply that $\deg(u)\le2(k-1)$, which is a contradiction since $\deg(u)=2k$. Hence $I(G)^{[k]}\subseteq P$ and $\height\,I(G)^{[k]}\leq n-k+1$.
		%To prove (ii), it is enough to notice that for any $1\le i\le n-k+1$, the monomial
		% $$
		% (x_iy_i)(\prod_{j=n-k+2}^n(x_jy_j))
		%  $$
		%  is a minimal monomial generator of $I(G)^{[k]}$ which is not contained in $P_i$. This shows that $I(G)^{[k]}\not\subseteq P_i$ and (ii) follows. Hence, $\height\,I(G)^{[k]}\leq n-k+1$.
	\end{proof}
	
	Using Theorem~\ref{Thm:I(G)[k]CM-perfect} and Proposition~\ref{Prop:dim-I(G)[k]-VWC} we are able to classify the edge ideals of very well-covered graphs whose all matching powers are Cohen-Macaulay.
	
	\begin{Corollary}\label{Thm:I(G)[k]CM-VWC}
		Let $G$ be a very well-covered graph. Then, the following conditions are equivalent.
		\begin{enumerate}
			\item[\textup{(a)}] $I(G)^{[k]}$ is Cohen-Macaulay, for all $1\le k\le\nu(G)$.
			\item[\textup{(b)}] $G$ is a Cohen-Macaulay forest.
		\end{enumerate}
	\end{Corollary}
	
	Bipartite graphs and  whisker graphs are examples of the very well-covered graphs.
	Recall that a graph $G$ is called {\em bipartite}, if there exists a partition $X\cup Y$ of the vertex set $V(G)$ such that any edge of $G$ is of the form $\{x,y\}$ for some $x\in X$ and some $y\in Y$. A {\em whisker graph} is a graph obtained by attaching a pendant edge to each vertex of a given graph.
	
	Corollary \ref{Thm:I(G)[k]CM-VWC} implies immediately
	
	\begin{Corollary}\label{bipartite}
		The only bipartite graphs $G$ for which $I(G)^{[k]}$ is Cohen-Macaulay for all $1\le k\le\nu(G)$ are the Cohen-Macaulay forests.
	\end{Corollary}
	\begin{Corollary}\label{whisker}
		The only whisker graphs $G$ for which $I(G)^{[k]}$ is Cohen-Macaulay for all $1\le k\le\nu(G)$ are the Cohen-Macaulay forests.
	\end{Corollary}
	
	Recall that the \textit{normalized depth function} $g_I$ of squarefree monomial ideal $I\subset S$ is defined by setting $g_I(k)=\depth\,S/I^{[k]}-(\alpha(I^{[k]})-1)$, for all $1\le k\le\nu(I)$, where $\alpha(I^{[k]})$ is the initial degree of $I^{[k]}$. Using Proposition~\ref{Prop:dim-I(G)[k]-VWC} we obtain
	
	\begin{Corollary}\label{NorDepthFun}
		Let $G$ be a very well-covered graph such that $I(G)^{[k]}$ is Cohen-Macaulay, for all $1\le k\le\nu(G)$. Then,
		$$
		g_{I(G)}(k)=|V(G)|/2-k,\,\,\,\,\,\textit{for all}\,\,\,\,1\le k\le\nu(G).
		$$
	\end{Corollary}
	
	\section{Chordal graphs and Cameron-Walker graphs}\label{sec3}
	
	We characterize chordal graphs and Cameron-Walker graphs whose matching powers of edge ideals are all Cohen-Macaulay. For chordal graphs we use a result by Herzog, Hibi and Zheng ~\cite{HHZ} which characterizes the Cohen-Macaulay chordal graphs. A {\em clique} of a graph $G$ is a subset $F\subseteq V(G)$ such that the induced subgraph of $G$ on $F$ is a complete graph. The {\em clique complex} of $G$
	is the simplicial complex on the vertex set $V(G)$ whose faces are the cliques of $G$, and we denote it by $\Delta(G)$. A {\em free vertex} of $\Delta(G)$ is a vertex which belongs to precisely one facet of $\Delta(G)$. 
	The main theorem in~\cite{HHZ} shows that $G$ is a Cohen-Macaulay chordal graph if and only if $V(G)$ is the disjoint union of the facets of  $\Delta(G)$ which admit a free vertex.

	We set $r_G$ to be the number of facets of $\Delta(G)$ which admit a free vertex.
	
	\begin{Theorem}\label{chordal}
		Let $G$ be a chordal graph.
		The following conditions are equivalent.
		\begin{enumerate}
			\item[\textup{(a)}] $I(G)^{[k]}$ is Cohen-Macaulay, for all $1\le k\le\nu(G)$.
			\item[\textup{(b)}] $G$ is a Cohen-Macaulay forest or a complete graph.
		\end{enumerate}
	\end{Theorem}
	
	For the proof of this theorem, we need some auxiliary lemmas.
	\begin{Lemma}\label{Lem:aux1}
		Let $S=K[x_1,\dots,x_n]$ and let $\m=(x_1,\dots,x_n)$ be the maximal homogeneous ideal. Then $\depth S/\m^{[k]}=\dim S/\m^{[k]}=k-1$ for all $1\le k\le n$.
	\end{Lemma}
	\begin{proof}
		Notice that $\m^{[k]}$ is the so-called squarefree Veronese ideal in the variables $x_1,\dots,x_n$ generated in degree $k$. This ideal is also denoted by $I_{n,k}$ in \cite{CF2024}. The assertion follows from \cite[Lemma 2]{CF2024} and the Auslander-Buchsbaum formula.
	\end{proof}
	\begin{Lemma}\label{Lem:aux2}
		Let $R=K[x_1,\dots,x_n]$, $T=K[y_1,\dots,y_m]$ and $S=R\otimes_KT=K[x_1,\dots,x_n,y_1,\dots,y_m]$. Let $I_1\subset R$ be the edge ideal of the complete graph on the vertex set $\{x_1,\dots,x_n\}$, $I_2\subset T$ be the edge ideal of the complete graph on the vertex set $\{y_1,\dots,y_m\}$, and set $I(G)=I_1+I_2$. Suppose that $n+m\ge5$. Then, there exists an integer $k\in\{\nu(G)-1,\nu(G)\}$ such that $I(G)^{[k]}$ is not Cohen-Macaulay.
	\end{Lemma}
	\begin{proof}
		Let $\m=(x_1,\dots,x_n)$ and $\n=(y_1,\dots,y_m)$. Then $I_1=\m^{[2]}$ and $I_2=\n^{[2]}$.
		
		If at least one of $n$ and $m$ is an odd integer, say $n$, then $I(G)^{[\nu(G)]}$ is not Cohen-Macaulay. Indeed, let $x_i\in V(K_m)$. Then neither $G$ nor $G\setminus\{x_i\}$ have perfect matchings, because $K_n$ is a component of both graphs with an odd number of vertices. Then, by Theorem \ref{Thm:I(G)[nu(G)]CM}(c), $I(G)^{[\nu(G)]}$ is not Cohen-Macaulay.
		
		Suppose now that both $n$ and $m$ are even integers, say $n=2r$ and $m=2t$. Then $n+m\ge6$ and we have
		\begin{equation}\label{eq:Betti-m-n}
			I(G)^{[\nu(G)-1]}=\m^{[2r]}\n^{[2t-2]}+\m^{[2r-2]}\n^{[2t]},
		\end{equation}
		with the convention that $J^{[0]}=S$ for a monomial ideal $J\subset S$. We claim that (\ref{eq:Betti-m-n}) is a Betti splitting. Indeed,
		$$
		(\m^{[2r]}\n^{[2t-2]})\cap(\m^{[2r-2]}\n^{[2t]})=\m^{[2r]}\n^{[2t]}.
		$$
		Since $\m$ and $\n$ are monomial ideals in disjoint sets of variables, it is clear that the inclusion maps $\m^{[2r]}\n^{[2t]}\rightarrow\m^{[2r]}\n^{[2t-2]}$ and $\m^{[2r]}\n^{[2t]}\rightarrow\m^{[2r-2]}\n^{[2t]}$ are $\Tor$-vanishing, if and only if, the inclusion maps $\n^{[2t]}\rightarrow\n^{[2t-2]}$ and $\m^{[2r]}\rightarrow\m^{[2r-2]}$ are $\Tor$-vanishing, which is clearly the case. Therefore, equation (\ref{eq:BettiSplitEq}) yields
		\begin{equation}\label{eq:depth-aux}
			\depth S/I(G)^{[\nu(G)-1]}\ =\ \min\left\{\substack{\displaystyle\depth S/\m^{[2r]}\n^{[2t-2]},\,\,\depth S/\m^{[2r-2]}\n^{[2t]}\\[5pt] \displaystyle\depth S/\m^{[2r]}\n^{[2t]}-1}\right\}.
		\end{equation}
		Again, because $\m$ and $\n$ are monomial ideals in disjoint sets of variables, by using Lemma \ref{Lem:aux1}, we obtain that
		$$
		\depth S/\m^{[2r]}\n^{[2t-2]}=\begin{cases}
			\depth R/\m^{[2r]}+\depth T/\n^{[2t-2]}+1=n+m-3&\textup{if}\ t>1,\\
			\depth R/\m^{[2r]}+\depth T=n+m-1&\textup{if}\ t=1.
		\end{cases}
		$$
		Similarly,
		$$
		\depth S/\m^{[2r-2]}\n^{[2t]}=\begin{cases}
			n+m-3&\textup{if}\ r>1,\\
			n+m-1&\textup{if}\ r=1.
		\end{cases}
		$$
		and
		$$
		\depth S/\m^{[2r]}\n^{[2t]}=\depth R/\m^{[2r]}+\depth T/\n^{[2s]}+1=n+m-1.
		$$
		Since $n+m\ge 6$, we have $r+t\ge 3$. Thus, at least one of the integers $r$ or $t$ is strictly bigger than one, and so, the previous three formulas combined with (\ref{eq:depth-aux}) yield that
		$$
		\depth S/I(G)^{[\nu(G)-1]}=n+m-3.
		$$
		We claim that $\dim S/I(G)^{[\nu(G)-1]}=n+m-2$, which finishes the proof. For this aim, it is enough to show that $\height\,I(G)^{[\nu(G)-1]}=2$. We claim that $I(G)^{[\nu(G)-1]}$ is contained in $(x_1,y_1)$. This is clear, as the following computation shows
		$$
		I(G)^{[\nu(G)-1]}\ =\ \begin{cases}
			(x_1\cdots x_n)\n^{[2t-2]}+(y_1\cdots y_m)\m^{[2r-2]}&\textup{if}\ r>1\ \textup{and}\ t>1,\\
			(x_1\cdots x_n)+(y_1y_2)\m^{[2r-2]}&\textup{if}\ r>1\ \textup{and}\ t=1,\\
			(x_1x_2)\n^{[2t-2]}+(y_1\cdots y_m)&\textup{if}\ r=1\ \textup{and}\ t>1.
		\end{cases}
		$$
		Since at least one of $r$ and $t$ is strictly bigger than one, the previous computation also shows that no principal monomial prime ideal can be an associated prime of $I(G)^{[\nu(G)-1]}$. Hence $\height\,I(G)^{[\nu(G)-1]}>1$. Consequently $(x_1,y_1)\in\Ass\,I(G)^{[\nu(G)-1]}$. This shows that $\height\,I(G)^{[\nu(G)-1]}=2$, which concludes the proof.
	\end{proof}
	
	\begin{proof}[Proof of Theorem \ref{chordal}]
		(a) $\Rightarrow$ (b) %Since any connected component of a chordal graph is again chordal, by Corollary~\ref{ConnCom}, we may assume that 
		Let $G$ be a Cohen-Macaulay chordal graph. If $G$ is a forest, we are done. So, we may assume that $G$ has a clique of size at least $3$. We need to prove that $r_G=1$. Indeed, we show by induction on $r_G$ that if $r_G\geq 2$, then  $I(G)^{[k]}$ is not Cohen-Macaulay for some integer $k$. 
		For the first step of the induction assume that $r_G=2$. 
		First we show that if $G$ is connected with $r_G=2$, then $I(G)^{[2]}$ is not Cohen-Macaulay. 
		Let $F_1,F_2$ be the facets of $\Delta(G)$ which admit a free vertex.
		Since $I(G)$ is 
		Cohen-Macaulay, by the characterization of Cohen-Macaulay chordal graphs~\cite{HHZ} mentioned in the beginning of this section, we have $V(G)$ is the disjoint union of $F_1$ and $F_2$. Since by our assumption $G$ is not a forest, we may assume that $|F_1|\geq 3$. Let $x_1\in F_1$ and $x_2\in F_2$ be free vertices.
		Since $G$ is connected by our assumption, we may choose $x_i,x_j\in F\setminus \{x_1\}$ and $x_k\in F_2\setminus \{x_2\}$ so that $\{x_i,x_k\}\in E(G)$.
		Then the prime ideals 
		$P=(x_{\ell}: x_{\ell}\in V(G)\setminus \{x_1,x_i,x_j,x_2\})$ and $Q=(x_{\ell}: x_{\ell}\in V(G)\setminus \{x_1,x_i,x_k\})$ are minimal prime ideals of $I(G)^{[2]}$.
		This shows that $I(G)^{[2]}$ is not unmixed and hence not Cohen-Macaulay.
		
		Suppose now that $r_G=2$ but $G$ is not connected. Then $G$ is the disjoint union of two complete graphs $K_{n}$ and $K_{m}$. We have $n+m\ge5$, otherwise $G$ would be a forest. Then, Lemma \ref{Lem:aux2} implies that $I(G)^{[k]}$ is not Cohen-Macaulay for some $k\in\{\nu(G)-1,\nu(G)\}$ .
		
		Now, let $r_G>2$. Assume inductively that for any Cohen-Macaulay chordal graph $H$ which is not a forest and with $2\leq r_H<r_G$, the ideal $I(H)^{[k]}$ is not Cohen-Macaulay for some $k$. Let $F_1,\ldots,F_{r_G}$ be the facets of $\Delta(G)$ which admit a free vertex. Since by our assumption $G$ is not a forest, we may assume that $|F_1|\geq 3$. %Since $G$ is connected, without loss of generality we may assume that $F_1$ have vertices adjacent to $F_2$.  
		Choose a free vertex  $x\in F_{r_G}$ and set $H=G\setminus N_G[x]$. Since $V(G)$ is the disjoint union of $F_1,\ldots,F_{r_G}$, we have $V(H)=V(G)\setminus F_{r_G}$ and   the facets of $\Delta(H)$ which admit a free vertex are $F_1,\ldots,F_{r_G-1}$. Therefore, $H$ is a Cohen-Macaulay chordal graph and $2\leq r_H=r_G-1<r_G$.
		Moreover, $H$  it is not a forest, since $|F_1|\geq 3$. So by our induction hypothesis  
		$I(H)^{[k]}$ is not Cohen-Macaulay for some $k$. Hence by
		Theorem~\ref{neighbourhood}, $I(G)^{[k]}$ is not Cohen-Macaulay for some $k$.
		\smallskip
		
		(b) $\Rightarrow$ (a) If $G$ is a Cohen-Macaulay forest, the claim holds by \cite[Corollary 3.6] {DRS24}. Now, let $G$ be a complete graph $K_n$. For any integer $1\le k\le\nu(G)$, the ideal $I(G)^{[k]}$ is the squarefree Veronese ideal $\m^{[2k]}$, where $\m=(x_1,\dots,x_n)$ is the graded maximal ideal of $S$. So it is Cohen-Macaulay by \cite[Theorem 4.2]{HH2006}.
	\end{proof}
	
	Next, we study the class of Cameron-Walker graphs. We recall some concepts which are used in the sequel. An {\em induced matching} of $G$ is a matching $M$ of $G$ such that for any two edges $e=\{a,b\}$ and $e'=\{c,d\}$ belonging to $M$ the induced subgraph of $G$ on $\{a,b,c,d\}$ consists of two disjoint edges $e$ and $e'$. We denote by $\im(G)$ the maximum size of an induced matching of $G$. 
	
	A graph $G$ is called a {\em Cameron-Walker graph} if $G$  is connected and $\nu(G)=\im(G)$. Cameron-Walker graphs are classified in \cite[Theorem 1]{CW} (see also~\cite{HHKO}), where it is shown that  $G$ is a Cameron-Walker graph if and only if $G$ is either a star graph, or a star triangle, or it consists of a connected bipartite graph with vertex partition $[n]\sqcup[m]$ such that there is at least one leaf edge attached to each vertex $i\in[n]$ and that there may be possibly some pendant triangles attached to each vertex $j \in [m]$. By a star triangle we mean a graph joining some triangles at one common vertex. 
	
	The following result is crucial for the characterization of Cameron-Walker graphs whose all matching powers are Cohen-Macaulay. 
	
	\begin{Proposition}\label{SunnyDay}
		Let $G$ be a Cameron-Walker graph such that $I(G)^{[\nu(G)]}$ is Cohen-Macaulay. Then $G$ is either a star triangle or a bipartite graph. 
	\end{Proposition}
	
	\begin{proof}
		Suppose that $G$ is neither a star triangle nor a bipartite graph. Then $G$ is a graph consisting of a connected bipartite graph with vertex partition $[n]\sqcup[m]$ such that there is at least one leaf attached to each vertex in $[n]$ and that there exists at least one pendant triangle $T$ attached to some vertex in $[m]$. Suppose that $G$ has a perfect matching. Since $\nu(G)=\im(G)$, we conclude that $G$ has a perfect matching which is an induced matching $M$. Thus any vertex of the triangle $T$ is an endpoint of an edge in the  matching $M$, which is not possible since $M$ is an induced matching of $G$. So, $G$ does not have a perfect matching. Consider some vertex $i\in[n]$. Then the pendant vertices of $G$ which are adjacent to $i$ are isolated vertices in the graph $G\setminus\{i\}$. Since there exists at least one of them, we conclude that $G\setminus\{i\}$ does not have a perfect matching. So, it follows from Theorem~\ref{Thm:I(G)[nu(G)]CM} that $I(G)^{[\nu(G)]}$ is not Cohen-Macaulay. 
	\end{proof}
	
	\begin{Theorem}\label{Cameron-Walker}
		Let $G$ be a Cameron-Walker graph. The following conditions are equivalent.
		\begin{enumerate}
			\item[\textup{(a)}] $I(G)^{[k]}$ is Cohen-Macaulay, for all $1\le k\le\nu(G)$.
			\item[\textup{(b)}] $G$ is either $K_2$ or $K_3$.
		\end{enumerate}
	\end{Theorem}
	
	\begin{proof}
		(b) $\Rightarrow$ (a) is obvious.\smallskip
		
		(a) $\Rightarrow$ (b) By Corollary~\ref{SunnyDay} we have that $G$ is either a star triangle or a bipartite graph. Since a star triangle is a chordal graph, by Theorem~\ref{chordal} and the assumption (a) we conclude that $G$ is a triangle $K_3$. Now, suppose that $G$ is bipartite. Then by Corollary~\ref{bipartite} we have that $G$ is a Cohen-Macaulay forest.
		Suppose that $G$ is not an edge. By  ~\cite[Theorem 1.3]{HHKO}, $G$ should consist of a connected bipartite graph with vertex partition $[n]\sqcup[m]$ such that there is exactly one leaf edge attached to each vertex $i\in[n]$ and that there is exactly one pendant triangle attached to each vertex $j\in[m]$. This is not possible, since $G$ is a forest.
	\end{proof}
	
	\section{Edge ideals of small graphs whose all matching powers are Cohen-Macaulay}\label{sec4}
	
	In this final section, we provide the classification of the edge ideals of graphs with a small number of vertices $(n\in\{2,3,4,5,6,7\})$ whose all matching powers are Cohen-Macaulay. This classification was obtained by using \textit{Macaulay2} \cite{GDS}, and the packages \texttt{NautyGraphs} \cite{MPNauty} and \texttt{MatchingPowers} \cite{FPack2}.
	
	Let $G$ be a graph on $n$ non-isolated vertices. We run our experiments in the polynomial ring $S=\mathbb{Q}[x_i:i\in V(G)]$ over the rational numbers $\mathbb{Q}$.
	
	For $n\in\{2,3,4\}$, the only graphs $G$ for which $I(G)^{[k]}$ is Cohen-Macaulay for all $1\le k\le\nu(G)$ are the complete graphs and the Cohen-Macaulay forests.
	
	Before proceeding further, notice that if $n$ is odd, there is no Cohen-Macaulay forest having $n$ non-isolated vertices, since any Cohen-Macaulay forest has a perfect matching, and thus it has an even number of vertices.
	
	For $n=5$, besides the complete graph, we also have the following graph $H$ and the cycle $C_5$,\medskip
	
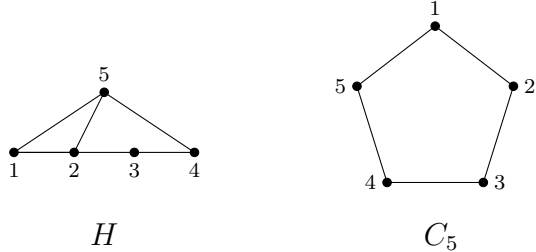
\begin{center}
		\begin{tikzpicture}[scale=0.8]
			\draw[-] (0,0) -- (1,0) -- (2,0) -- (3,0);
			\draw[-] (1.5,1) -- (0,0) -- (1,0) -- (1.5,1) -- (3,0);
			\filldraw[black] (0,0) circle (2pt) node[below]{{\tiny$1$}};
			\filldraw[black] (1,0) circle (2pt) node[below]{{\tiny$2$}};
			\filldraw[black] (2,0) circle (2pt) node[below]{{\tiny$3$}};
			\filldraw[black] (3,0) circle (2pt) node[below]{{\tiny$4$}};
			\filldraw[black] (1.5,1) circle (2pt) node[above]{{\tiny$5$}};
			\filldraw (7,2.1) circle (2pt) node[above]{{\tiny$1$}};
			\filldraw (8.3,1.1) circle (2pt) node[right]{{\tiny$2$}};
			\filldraw (7.8,-0.5) circle (2pt) node[right]{{\tiny$3$}};
			\filldraw (6.2,-0.5) circle (2pt) node[left]{{\tiny$4$}};
			\filldraw (5.7,1.1) circle (2pt) node[left]{{\tiny$5$}};
			\draw[-] (7,2.1)--(8.3,1.1)--(7.8,-0.5)--(6.2,-0.5)--(5.7,1.1)--(7,2.1);
			\filldraw (1.5,-1) node[below]{$H$};
			\filldraw (7.1,-1) node[below]{$C_5$};
		\end{tikzpicture}
	\end{center}
	
	For $n=6$, besides the complete graph and the Cohen-Macaulay forests, we only have two extra, non isomorphic, graphs. It can be noticed that for these two extra graphs $G_1,G_2$ we have that $I(G_i)^{[k]}=\m^{[2k]}$, for all $2\le k\le\nu(G_i)$ and $i=1,2$, are the squarefree Veronese ideals of degree $2k$, where $\m=(x_1,x_2,\dots,x_6)$.
	
	For $n=7$, besides the complete graph, there is again only two extra graphs $H_1,H_2$ for which $I(H_i)^{[k]}$ is Cohen-Macaulay for all $k$ and $i=1,2$. Once again, we have $I(H_i)^{[k]}=\m^{[2k]}$ for $i=1,2$ and $2\le k\le\nu(H_i)$, where $\m=(x_1,x_2,\dots,x_7)$.\smallskip
	
	It is well-known that up to $7$ vertices, the Cohen-Macaulay property of $I(G)$ is independent from the field $K$ of $S=K[x_i:i\in V(G)]$. Note that the Cohen-Macaulay property of the matching powers of complete graphs and Cohen-Macaulay forest is also independent from $K$. Moreover, for the extra graphs $G\in\{H,C_5,G_1,G_2,H_1,H_2\}$ we discovered above, it turns out that $I(G)^{[k]}=\m^{[2k]}$ for all $2\le k\le\nu(G)$. Since the Cohen-Macaulay property of a squarefree Veronese ideal does not depend on $K$, it follows that for the extra graphs $G\in\{H,C_5,G_1,G_2,H_1,H_2\}$, we have that $I(G)^{[k]}$ is Cohen-Macaulay for all $k$ and every field $K$. These facts suggest that
	\begin{Conjecture}
		The family of graphs $G$ for which $I(G)^{[k]}$ is Cohen-Macaulay for all $1\le k\le\nu(G)$ does not depend on the characteristic of $K$.
	\end{Conjecture}
	
	The above experiments also suggest that the following question could have a positive answer.
	\begin{Question}
		It is true that for all $n\ge 5$, up to isomorphism, there are only two graphs on $n$ non-isolated vertices, non isomorphic to the complete graph $K_n$ and non isomorphic to a Cohen-Macaulay forest, for which all the matching powers of their edge ideals are Cohen-Macaulay?
	\end{Question}\bigskip
	
	\noindent\textbf{Acknowledgment.}
	A. Ficarra was partly supported by INDAM (Istituto Nazionale di Alta Matematica), and also by the Grant JDC2023-051705-I funded by
	MICIU/AEI/10.13039/501100011033 and by the FSE+. Somayeh Moradi is supported by the Alexander von Humboldt Foundation.

\end{document}